\documentclass[12pt]{amsart}

\usepackage{amssymb}
\usepackage{amsthm}
\usepackage{amsmath}
\usepackage{graphicx}
\usepackage{comment}
\usepackage[usenames,dvipsnames]{xcolor}

 \newtheorem{thm}{Theorem}[section]
 
 \newtheorem{lem}[thm]{Lemma}

 \theoremstyle{definition}

 \theoremstyle{remark}
 \newtheorem*{rem}{Remark}

 \numberwithin{equation}{section}

\newcommand{\la}{\langle}
\newcommand{\ra}{\rangle}
\newcommand{\jbtt}{\la t/2 \ra}
\newcommand{\jbt}{\la t\ra}
\newcommand{\jbtr}{\la t-r\ra}
\newcommand{\rr}{{\mathbb R}}
\newcommand{\tr}{\mbox{tr}\;}

\begin{document}

\title[2-D Incompressible
  Elastodynamics]{Almost Global Existence for \\2-D Incompressible Isotropic
  Elastodynamics}

\author{Zhen Lei}
\address{School of
Mathematical Sciences\\
 LMNS and Shanghai Key
  Laboratory for Contemporary Applied Mathematics\\
   Fudan  University\\
Shanghai 200433\\
   P. R.\ China }
\email{leizhn@gmail.com, zlei@fudan.edu.cn}

 \author{Thomas C.\ Sideris}
\address{Department of Mathematics\\
University of California\\
Santa Barbara, CA 93106\\USA} \email{sideris@math.ucsb.edu}

 \author{Yi Zhou}
 \address{School of
Mathematical Sciences\\
 LMNS and Shanghai Key
  Laboratory for Contemporary Applied Mathematics\\
   Fudan  University\\
Shanghai 200433\\
   P. R.\ China}
\email{ yizhou@fudan.ac.cn}
\date{\today}

\thanks{The authors would like to thank Professor Fang-hua Lin of the Courant Institute for some
helpful discussions. }

\thanks{Zhen Lei was
supported by NSFC (grant No.11171072), FANEDD, Innovation Program of
Shanghai Municipal Education Commission (grant No.12ZZ012) and
NTTBBRT of NSF (No. J1103105).}

\thanks{Thomas C.\ Sideris was partially supported by the
National Science Foundation.
}

\thanks{Zhen Lei and Yi Zhou were
supported by the Foundation for Innovative Research Groups of NSFC
(grant No.11121101) and SGST 09DZ2272900.}


\begin{abstract}
We consider the Cauchy problem for 2-D incompressible isotropic
elastodynamics. Standard energy methods yield local solutions on a
time interval $[0,{T}/{\epsilon}]$, for initial data of the form
$\epsilon U_0$, where $T$ depends only on some Sobolev norm of
$U_0$. We show that for such data there exists a unique solution on
a time interval $[0, \exp{T}/{\epsilon}]$, provided
that $\epsilon$ is sufficiently small. This is
achieved by careful consideration of the structure of the
nonlinearity.
 The incompressible
elasticity equation is inherently linearly degenerate in the
isotropic case; in other words, the equation  satisfies a null
condition.  This is essential for time decay estimates.  The pressure, which
arises as a Lagrange multiplier to enforce the incompressibility
constraint, is estimated in a novel way as a nonlocal nonlinear term with
null structure.  The proof employs the generalized energy method of Klainerman,
enhanced by weighted $L^2$ estimates and the ghost weight introduced by Alinhac.
\end{abstract}

\maketitle

\section{Introduction}

The long-time behavior of elastic waves for isotropic
incompressible materials is studied in 2-D. The
equations of incompressible elastodynamics
display a linear
degeneracy in the isotropic case; i.e., the equation inherently
satisfies a null condition. By taking the advantage of this
 structure, we prove that the 2-D incompressible
isotropic nonlinear elastic system is almost globally well-posed
for small initial data. More precisely, we prove that for initial data
of the form $\epsilon U_0$, there exists a unique solution for a time interval
$[0, \exp({T(U_0)}/{\epsilon})]$, where $T(U_0)$ depends only on a certain weighted Sobolev
norm of the  $U_0$.

To place our result in context, we review a few highlights from the existence theory of nonlinear wave equations
and elastodynamics.
The initial value problem for small solutions of 3-D quasilinear wave equations with quadratic nonlinearities is
almost globally well-posed \cite{JohnKlainerman84}, and in general this is sharp \cite{John81}.  If, in addition, the nonlinearity satisfies the null condition, global existence
was shown independently in \cite{Christodoulou86} using conformal compactification
and in \cite{Klainerman86} using the generalized energy method.
The generalized energy method of Klainerman can be adapted to prove almost global existence for certain nonrelativistic systems of 3-D nonlinear wave equations,
 using scaling invariance to get weighted $L^2$-estimates, as
was first done in \cite{KS96}.
This approach was subsequently developed  to obtain global existence
under the null condition
 in \cite{SiderisTu01}, see also \cite{Yokoyama} for a different method.
A unified treatment for obtaining  weighted $L^2$ estimates for certain hyperbolic systems appeared in \cite{ST06}.

The existence question is more delicate in 2-D because, even with
the null condition, quadratic nonlinearities have critical decay. A
series of articles considered the case of cubically nonlinear
equations satisfying the null condition, see for example
\cite{Hoshiga,Katayama,L95}. Alinhac
\cite{Alinhac00} was the first to establish global existence for
null bilinear forms.
  His argument combines  vector fields  with what he called the ghost energy method, but crucially it
  also relies on finite propagation speed through a certain Hardy-type inequality.

 The long-time existence theory for isotropic elastodynamics largely follows the paradigm of nonlinear wave equations.
 Almost global existence of small displacement solutions for 3-D compressible elastodynamics was first shown in  \cite{John88},
 and counter-examples to global existence appeared in \cite{John84,T98}.
 The almost global existence proof was simplified considerably in \cite{KS96} by enhancing the vector field approach with
 a new weighted $L^2$ estimate based on scaling invariance to compensate for the absence of Lorentz invariance.
 In \cite{Sideris96}, it was first noticed that there was a  null structure compatible with the system
 of isotropic elastodynamics which can  be used to establish global existence of small displacement solutions in 3-D.
More comprehensive versions of this result appeared in \cite{Sideris00,Agemi00}.
 Compressible isotropic elastodynamics (in free space) is characterized by two wave families:  fast pressure waves and slower
 shear waves.  The aforementioned null condition limits the self-interaction of pressure waves.  The equations possess an inherent null
 structure for shear waves.  Thus, in the incompressible case where pressure waves are not present, it
 is plausible to expect global existence of small displacement solutions in 3-D without an additional null condition assumption,
 although the absence of finite propagation speed  presents an obstacle.
 Nevertheless,
 this intuition was confirmed in \cite{ST05,ST07}.

 Our method for proving almost global existence for incompressible isotropic
 elastodynamics in 2-D is based on  the ideas of \cite{ST07}, with two new ingredients:
 the treatment of the pressure term and the use of the ghost weight.
 The equations are written as a first-order system with constraints whose unknowns are
   the material space-time gradient of the deformation, expressed in spatial coordinates (Section \ref{results}).
The advantage of this choice is that the Lagrange multiplier which enforces the incompressibility constraint
appears as a pressure  gradient.

The generalized energy method  forms the backbone of the argument.  This enables control in $L^2(\rr^2)$ of derivatives
 of the solution with respect to the vector fields which arise as  infinitesimal generators corresponding to the fundamental invariance of
 the equations under translation, rotation, and scaling.  This invariance gives rise to the basic commutation identities for the vector fields
 (Section \ref{commutation}).
 The rotational vector field yields $|x|^{-1/2}$ spatial decay, by means of Sobolev-type inequalities (Section \ref{ellinfty}),
 which is weaker than the fully Lorentz invariant case.
 Using scaling invariance, the spatial decay
  is improved  to $t^{-1/2}$ time decay in $L^\infty(\rr^2)$ by means of a series of weighted estimates for the gradient of the
 solution which follow from an algebraic manipulation of the  equations and the constraints.
 This algebraic procedure (which is  an implementation of the method of \cite{ST06}) allows
 control of the solution gradient in $L^2(\rr^2)$ with the weight $(t-|x|)^{-1}$, and it also gives  $t^{-1}$ decay in $L^2(\rr^2)$ of
 certain special quantities (Sections \ref{special1}, \ref{special2}).
 For small energy solutions, the weighted estimates are  closed by a bootstrapping argument (Section \ref{weight}).
 The pressure gradient is estimated as a solution of a nonlinear Poissson equation (Section \ref{pressure}).
 The nonlinearities of this elliptic equation have a null structure.  Thus, the pressure term is essentially treated
 in a novel way as a nonlocal null form.  This absence of finite propagation speed prevents us from achieving
 global existence.
 Energy estimates are performed using the ghost weight of \cite{Alinhac00} (Section \ref{ghostenergy})
 which provides a convenient solution to the technical problem that the weighted estimates hold only for the
 solution gradient.

 Finally, the general isotropic case is easily treated by our method since it can be regarded as a higher order nonlinear correction
 to the Hookean case (Section \ref{general}).

Before ending the introduction, let us mention some related works on
viscoelasticity, where there is viscosity in the momentum equation.
In 2-D, the global well-posedness with near equilibrium states is
due to  \cite{LLZ2005} (see also \cite{LZ2005}), and in 3-D,
\cite{CZ2006} and \cite{LeiLZ08}, independently. The initial
boundary value problem is considered in \cite{LL2008}, the
compressible case can be found in \cite{QZ2010, HW2011}. For more
results near equilibrium,  readers are referred to
{\cite{HX2010,Qian2010,Qian2011,ZF-1,ZF-2}. In \cite{Lei} a class of large solutions
in two space dimensions are established via the strain-rotation
decomposition (which is based on earlier results in
\cite{Lei2008} and \cite{Lei3}). In all of these works, the initial
data is restricted by the viscosity parameter. The work
\cite{Kessenich} was the first to establish global existence for 3-D incompressible
viscoelastic materials uniformly in the viscosity parameter.

\section{Preliminaries and Main Results}
\label{results}

Classically, the motion of an elastic body is described as a
second-order evolution equation in Lagrangian coordinates. In the
incompressible case, the equations are more conveniently written as a first-order system
with constraints in Eulerian coordinates. We start with a
time-dependent family of orientation-preserving diffeomorphisms
$x(t, \cdot)$, $0 \leq t < T$. Material points $X$ in the
reference configuration are deformed to the spatial positions
$x(t, X)$ at time $t$. Derivatives with respect to the spatial
coordinates will be written as $(\partial_t, \nabla) = \partial$.
Let $X(t,x)$ be the corresponding reference map:  $X(t,\cdot)$
is the inverse of $x(t,\cdot)$.

\begin{lem}\label{lem2-1}
Given a family of  deformations $x(t, X)$,
define the velocity,
deformation gradient, and displacement gradient as follows:
\begin{equation}\nonumber
v(t, x) = \frac{dx(t, X)}{dt}\Big|_{X = X(t, x)},\quad F(t, x) =
\frac{\partial x(t, X)}{\partial X}\Big|_{X = X(t, x)},\quad G(t,
x) = F(t, x) - I.
\end{equation}
Then for $0 \leq t < T$,  we
have
\begin{equation}
\label{derconstr}
\partial_jG_{ik} - \partial_kG_{ij}
= G_{mk}\partial_mG_{ij} - G_{mj}\partial_mG_{ik},
\quad i, j, k \in \{1, 2, \cdots, n\},
\end{equation}
and
if $x(t,X)$ is incompressible, that is $\det F(t, x) \equiv 1$,
then
\begin{equation}
\label{incomconstr}
\nabla\cdot v=\partial_iv_i=0
\quad\mbox{and}\quad
(\nabla\cdot F^\top)_j=\partial_iF_{ij} = 0.
\end{equation}
\end{lem}
\begin{proof}
The first statement expresses the commutativity of material derivates $D_jD_k=D_kD_j$
in spatial coordinates.  The second statement follows from Nanson's formula.
Details are given in \cite{LeiLZ08}.
\end{proof}
Here and in what follows, we use the
summation convention over repeated indices. The identities
\eqref{derconstr} and \eqref{incomconstr} will be used repeatedly
in the sequel.

To best illustrate our methods and ideas, we shall  first consider
the equations of motion for incompressible Hookean elasticity,
which corresponds to the Hookean strain energy function $W(F) =
\frac{1}{2}|F|^2$.
\begin{equation}\label{Elas}
\begin{cases}
\partial_tv + v\cdot\nabla v + \nabla p = \nabla\cdot(FF^\top),\\[-4mm]\\
\partial_tF + v\cdot\nabla F = \nabla vF,\\[-4mm]\\
\nabla\cdot v = 0,\quad \nabla\cdot F^\top=0.
\end{cases}
\end{equation}
As will be seen in Section \ref{general} where the case of
general energy function is discussed, there is no essential loss
of generality in considering this simplest case.
Since the 3-D case has been treated in \cite{ST05,
ST07}, below we will focus on 2-D.

Denote the
rotation operator by
\begin{equation}\nonumber
\Omega = x_2\partial_1 - x_1\partial_2=x^\perp\cdot\nabla,
\end{equation}
and the scaling operator by
\begin{equation}\nonumber
S = t\partial_t + x_1\partial_1 + x_2\partial_2=t\partial_t+r\partial_r,\quad S_0 =
r\partial_r.
\end{equation}
We shall frequently use the decomposition
\begin{equation}
\label{der-decomp}
\nabla=(x/r)\partial_r+(x^\perp/r^2)\Omega.
\end{equation}
Let $\Gamma$ be any of the following differential operators
\begin{equation}\label{operator}
\Gamma \in \{\partial_t, \partial_1, \partial_2,
\widetilde{\Omega}, S\}.
\end{equation}
As in previous works, we define
\begin{equation}\nonumber
\begin{cases}
\widetilde{\Omega}f = \Omega f,\quad \forall\ {\rm scalar}\ f,\\[-4mm]\\
\widetilde{\Omega}v = \Omega v + (e_2\otimes e_1 - e_1\otimes
e_2)v,\quad \forall\ v \in \mathbb{R}^2,\\[-4mm]\\
\widetilde{\Omega}G = \Omega G + [(e_2\otimes e_1 - e_1\otimes
e_2), G],\quad \forall\ G \in \mathbb{R}^2\otimes \mathbb{R}^2,
\end{cases}
\end{equation}
where $[A,B] = AB - BA$ denotes the standard Lie bracket product.

Define the
generalized energy by
\begin{equation}\label{Energy}
E_k(t) = \sum_{|\alpha| \leq k}\big(\|\Gamma^\alpha v(t,
\cdot)\|_{L^2}^2 + \|\Gamma^\alpha G(t, \cdot)\|_{L^2}^2\big).
\end{equation}
We also define the weighted energy norm
\begin{equation}\label{WEnergy}
X_k(t) = \sum_{|\alpha| \leq k - 1}\big(\|\jbtr \nabla\Gamma^\alpha v\|_{L^2}
+\|\jbtr \nabla\Gamma^\alpha G\|_{L^2}\big),
\end{equation}
in which we denote $\la \sigma\ra = \sqrt{1 + \sigma^2}$.

To describe the space of the initial data, we introduce
\begin{equation}\nonumber
\Lambda = \{\nabla, S_0, \Omega\},
\end{equation}
and
\[
H^k_\Lambda =\{f:\sum_{|\alpha|\le k}\|\Lambda^\alpha f\|_{L^2}<\infty\}.
\]
Then we define
\[
H^k_\Lambda(T)=\{(v,G):[0,T)\to \rr^2\times(\rr^2\otimes \rr^2): (v,G)\in\cap_{j=0}^k C^j([0,T);H^{k-j}_\Lambda\}
\]
with the norm
\[
\sup_{0\le t<T} E_k(t)^{1/2}.
\]

The main result of this paper is

\begin{thm}\label{thm}
Let $(v_0,G_0)\in H^k_\Lambda$, with $k\ge5$.
Suppose that $(v_0,G_0)$ satisfy the constraints \eqref{derconstr}, \eqref{incomconstr}
and $\|(v_0,G_0)\|_{H^k_\Lambda}<\epsilon$.

 There exist two
positive constants $C_0$ and $\epsilon_0$ which depend only on $k$
such that, if $\epsilon \leq \epsilon_0$, then the system of
incompressible Hookean elastodynamics \eqref{Elas} with initial data
$(v_0, F_0)=(v_0,I+G_0)$ has a unique solution $(v,F)=(v, I+G)$,
with $(v, G) \in H^k_\Lambda(T)$,  $T\ge
\exp(C_0/\epsilon)$ and $E_k(t)\le 2\epsilon^2$, for $0\le t<T$.
\end{thm}

\section{$L^\infty$ Decay Estimate}
\label{ellinfty}

In this section we derive several weighted $L^\infty-L^2$ Sobolev imbedding inequalities.
These will be useful in proving decay of solutions in $L^\infty$.

\begin{lem}\label{rgsob}
For all radial functions $f\in H^1(\mathbb{R}^2)$, $\lambda=1,2$, there holds
\begin{align}
\label{rgsob0}
r^\lambda|f(r)|^2&\lesssim \|r^{\lambda-1}\partial_rf\|_{L^2}^2+\|f\|_{L^2}^2\\
\label{rgsob1}
r\langle t-r\rangle^\lambda|f(r)|^2&\lesssim \|\langle t-r\rangle\partial_rf\|_{L^2}^2+\|\jbtr^{\lambda-1}f\|^2_{L^2},
\end{align}
provided that the right-hand side is finite.
\end{lem}

\begin{proof}
It suffices to show the lemma for $f \in
C_0^1(\mathbb{R}^2)$ and then use a completion argument.

The first inequality is shown as follows
\begin{align*}
r^\lambda|f(r)|^2&=-r^\lambda \int_r^\infty\partial_\rho[|f(\rho)|^2]d\rho\\
&=-r^\lambda\int_r^\infty 2f(\rho)\partial_rf(\rho)d\rho\\
&\lesssim \int_r^\infty\rho^{\lambda-1}|f(\rho)||\partial_rf(\rho)|\rho d\rho\\
&\lesssim \|r^{\lambda-1}\partial_rf\|_{L^2}\|f\|_{L^2}.
\end{align*}

The calculation for the other inequality is similar.
\begin{align*}
r\langle t-r\rangle^\lambda|f(r)|^2  &= -
  r\int_r^\infty
  \partial_\rho[\langle t-\rho\rangle^\lambda |f(\rho)|^2] d\rho\\
&\lesssim r\int_r^\infty
[\langle t-\rho\rangle^\lambda  |f(\rho)||\partial_r
  f(\rho)|+\langle t-\rho\rangle^{\lambda-1} |f(\rho)|^2]  d\rho\\
&\lesssim \int_r^\infty
[\langle t-\rho\rangle^2  |\partial_r
  f(\rho)|^2+\langle t-\rho\rangle^{2(\lambda-1)} |f(\rho)|^2]  \rho d\rho\\
  &= \|\jbtr\partial_rf\|_{L^2}^2+\|\jbtr^{\lambda-1} f\|_{L^2}^2.
\end{align*}
\end{proof}

The next lemma is just the Sobolev imbedding theorem $H^1
\hookrightarrow L^\infty$ on the circle $\mathbb{S}^1$.

\begin{lem}\label{lem3-4}
For each $x = r\omega \in \mathbb{R}^2$ and $f(r\omega) \in
H^1(\mathbb{S}^1)$, there holds
\begin{equation}\nonumber
|f(r\omega)|^2 \lesssim \sum_{a = 0,
1}\int_{\mathbb{S}^1}|\Omega^a f(r\omega)|^2d\sigma.
\end{equation}
\end{lem}
\begin{proof}
Write $\omega = (\cos\theta, \sin\theta)$. We have
\begin{align*}
|f(r\omega)|^2 &= |f(r\omega)|^2\cos^2\theta +
  |f(r\omega)|^2\sin^2\theta\\
&= \int_{\frac{\pi}{2}}^\theta\frac{d}{d\phi}\big[|f(r\cos\phi,
  r\sin\phi)|^2\cos^2\phi\big] d\phi\\
&\quad +\ \int_{0}^\theta\frac{d}{d\phi}\big[|f(r\cos\phi,
  r\sin\phi)|^2\sin^2\phi\big] d\phi\\
&\lesssim \int_{\mathbb{S}^1}[|\Omega f(r\omega)||f(r\omega)|+|f(r\omega)|^2]d\sigma,
\end{align*}
from which  the result follows by  the Cauchy-Schwarz inequality.
\end{proof}

These two results combine to yield
\begin{lem}\label{gsob}
For all $f\in H^2(\rr^2)$, $\lambda=1,2$, there holds
\begin{align}
\label{gsob0}
r^{\lambda}|f(x)|^2&\lesssim \sum_{a=0,1}\|r^{\lambda-1}\partial_r \Omega^a f\|_{L^2}^2+\|\Omega^a f\|_{L^2}^2\\
\label{gsob1}
r\la t-r\ra^\lambda|f(x)|^2&\lesssim \sum_{a=0,1}[\|\la t-r\ra \partial_r \Omega^a f\|_{L^2}^2+\|\jbtr^{\lambda-1}\Omega^a f\|_{L^2}^2],
\end{align}
provided the right-hand side is finite.
\end{lem}

\begin{proof}
First apply Lemma \ref{lem3-4} to the left-hand side, then apply Lemma \ref{rgsob} to the integrand.
\end{proof}

\begin{rem}
The estimates of Lemma \ref{gsob} hold in the vector- and matrix-valued cases
using $\widetilde\Omega$ in place of $\Omega$.  This can be seen by applying Lemma \ref{gsob} component-wise
and the using the fact that,  for example, $|\Omega v|
\le |\widetilde\Omega v|+|v|$, for vectors $v$.
\end{rem}

\begin{rem}
The case $\lambda=2$ in \eqref{gsob1} will be applied only to derivatives $\nabla f$.
\end{rem}

\begin{lem}
\label{standard-sob}
For all $f\in H^2(\rr^2)$,  there holds
\[
\jbt \| f\|_{L^\infty(r\le\jbtt)}\lesssim \sum_{|\alpha|\le2}\|\jbtr\partial^\alpha f\|_{L^2},
\]
provided the right-hand side is finite.
\end{lem}
\begin{proof}
Let $\varphi\in C_0^\infty$, satisfy $\varphi(s)=1$ for $s\le 1$, $\varphi(s)=0$ for $s\ge3/2$.
Note that $\jbt\lesssim\jbtr$ on $\mbox{supp}\;\varphi(r/\jbtt)$.  Thus, for $|x|\le\jbtt$, we have
by Sobolev imbedding $H^2(\rr^2)\hookrightarrow L^\infty(\rr^2)$ that
\begin{align*}
\jbt|f(x)|
&=\jbt\varphi(r/\jbtt)|f(x)|\\
&\lesssim \jbt  \|\varphi(r/\jbtt)f\|_{H^2}\\
&\lesssim \jbt\sum_{|\alpha|\le2}\|\partial^\alpha f\|_{L^2(r\le \frac32\jbtt)}\\
&\lesssim \sum_{|\alpha|\le2}\|\jbtr\partial^\alpha f\|_{L^2}.
\end{align*}
\end{proof}

\begin{rem}
This result will only be applied to derivatives $\nabla f$.
\end{rem}

\section{Commutation}
\label{commutation}
For any multi-index $\alpha$, we have the following commutation
properties when applying $\Gamma^\alpha$ to the equation
\eqref{Elas} (see, for instance, \cite{ST07, Kessenich} for
details).  This will be essential in all of the subsequent estimations.
\begin{equation}\label{1}
\begin{cases}
\partial_t\Gamma^\alpha v - \nabla\cdot\Gamma^\alpha G =  - \nabla\Gamma^\alpha p + f_\alpha, \\[-4mm]\\
\partial_t\Gamma^\alpha G  - \nabla\Gamma^\alpha v = g_\alpha,\\[-4mm]\\
\nabla\cdot\Gamma^\alpha v = 0,\quad \nabla\cdot\Gamma^\alpha G^\top
= 0,
\end{cases}
\end{equation}
where
\begin{equation}\label{2-5}
\begin{cases}
f_\alpha =
  \sum\limits_{\beta + \gamma = \alpha}[-\Gamma^\beta v\cdot\nabla \Gamma^\gamma v
  + \nabla\cdot(\Gamma^\gamma G \Gamma^\beta G^\top)],\\[-4mm]\\
g_\alpha = \sum\limits_{\beta + \gamma = \alpha}
  [-\Gamma^\beta v\cdot\nabla \Gamma^\gamma
  G + \nabla\Gamma^\gamma v \Gamma^\beta G],
\end{cases}
\end{equation}
From \eqref{derconstr}, we also have
\begin{equation}\label{1-1}
\nabla^\perp\cdot \Gamma^\alpha G
= h_\alpha,
\end{equation}
where
\begin{equation}\label{2-7}
(h_\alpha)_i = \sum_{\beta + \gamma = \alpha} \big[\Gamma^\beta
G_{m1}\partial_m\Gamma^\gamma G_{i2} - \Gamma^\beta
G_{m2}\partial_m\Gamma^\gamma G_{i1}\big].
\end{equation}

\section{Bound for the Pressure Gradient}
\label{pressure}
The following Lemma shows that the pressure gradient
may be treated as a nonlinear term.
The first estimate appeared
in \cite{ST07}.  The second is a
 novel refinement which saves one
derivative over the first bound  and  which allows us to exploit the null structure.  This is essential
in Section \ref{ghostenergy} when we estimate the ghost weighted energy.

\begin{lem}\label{lem4-1}
Let $(v,F)=(v,I+G)$, $(v, G) \in H^k_\Lambda(T)$,  solve the equation \eqref{Elas}
and  the constraint \eqref{derconstr}.
Then we have
\begin{align}
\label{pressure-1}
&\|\nabla\Gamma^\alpha p\|_{L^2} \lesssim \| f_\alpha\|_{L^2}\\
\label{pressure-2}
&\|\nabla\Gamma^\alpha p\|_{L^2} \lesssim
\sum_{\tiny\begin{matrix}\beta+\gamma=\alpha\\ |\beta|\le|\gamma|\end{matrix} }
\|\partial_j\Gamma^\beta v_i\Gamma^\gamma v_j
-\partial_j\Gamma^\beta G_{ik}\Gamma^\gamma G_{jk}\|_{L^2},
\end{align}
for all $|\alpha|\le k-1$.
\end{lem}
\begin{proof}
Taking the divergence of the first equation of \eqref{1} and then using the constraints given in third equation
of \eqref{1}, we find
\[
\Delta\Gamma^\alpha p = \nabla\cdot f_\alpha +
\nabla\cdot(\nabla\cdot \Gamma^\alpha G) - \partial_t\nabla\cdot
\Gamma^\alpha v=\nabla\cdot f_\alpha.
\]
By \eqref{2-5} and the constraint equations in \eqref{1}, we have
\begin{align*}
\nabla\cdot f_\alpha
&=-\sum_{\beta+\gamma=\alpha }\partial_i\partial_j(\Gamma^\beta v_i\Gamma^\gamma v_j
-\Gamma^\beta G_{ik}\Gamma^\gamma G_{jk})\\
&=-\sum_{\tiny\begin{matrix}\beta+\gamma=\alpha\\|\beta|\le|\gamma|\end{matrix} }\partial_i\partial_j(\Gamma^\beta v_i\Gamma^\gamma v_j
-\Gamma^\beta G_{ik}\Gamma^\gamma G_{jk})\\
&\qquad -\sum_{\tiny\begin{matrix}\beta+\gamma=\alpha\\|\beta|>|\gamma|\end{matrix} }
\partial_i\partial_j(\Gamma^\beta v_i\Gamma^\gamma v_j
-\Gamma^\beta G_{ik}\Gamma^\gamma G_{jk})\\
&=-\sum_{\tiny\begin{matrix}\beta+\gamma=\alpha\\|\beta|\le|\gamma|\end{matrix} }
\partial_i(\partial_j\Gamma^\beta v_i\Gamma^\gamma v_j
-\partial_j\Gamma^\beta G_{ik}\Gamma^\gamma G_{jk})\\
&\qquad -\sum_{\tiny\begin{matrix}\beta+\gamma=\alpha\\|\beta|>|\gamma|\end{matrix} }
\partial_j(\Gamma^\beta v_i\partial_i\Gamma^\gamma v_j
-\Gamma^\beta G_{ik}\partial_i\Gamma^\gamma G_{jk}).
\end{align*}
The result now follows since
\[
\nabla\Gamma^\alpha p=\Delta^{-1}\nabla(\nabla\cdot f_\alpha)
\]
and since $\Delta^{-1}\nabla\otimes\nabla$ is bounded in $L^2$.

\end{proof}

\section{ Estimates for Special Quantities, I}
\label{special1}
\begin{lem}
\label{spellinfty}
Suppose that $(v,F)=(v,I+G)$,
$(v, G) \in H^k_\Lambda(T)$,  solves \eqref{Elas}, \eqref{derconstr}.
Define
\[
L_k=\sum_{|\alpha|\le k}[|\Gamma^\alpha v|+|\Gamma^\alpha G|]
\]
and
\begin{equation}
\label{ndef}
N_k=\sum_{|\alpha|\le k-1}[t|f_\alpha|+t|g_\alpha|+(t+r)|h_\alpha|+t|\nabla \Gamma^\alpha p|].
\end{equation}
Then for all $|\alpha|\le k-1$,
\begin{align}
\label{spconv}
&r|\partial_r\Gamma^\alpha v\cdot\omega|\lesssim L_k\\
\label{spcong}
&r|\partial_r \Gamma^\alpha G^\top\omega|\lesssim L_k\\
\label{spdivg}
&r|\partial_r \Gamma^\alpha G\omega-\nabla \cdot\Gamma^\alpha G|\lesssim L_k\\
\label{spcurlg}
&r|\partial_r\Gamma^\alpha G\omega^\perp|\lesssim L_k+N_k\\
\label{sppde}
&(t\pm r)|\nabla\Gamma^\alpha v\pm \nabla\cdot\Gamma^\alpha G\otimes\omega|
\lesssim L_k+N_k.
\end{align}
\end{lem}

\begin{proof}
By the decomposition  \eqref{der-decomp} we have,
\begin{equation}
\label{gradvident}
\nabla \Gamma^\alpha v = \partial_r \Gamma^\alpha v \otimes\omega+\frac1r\Omega\Gamma^\alpha v\otimes\omega^\perp.
\end{equation}
Taking the trace of this identity yields
\[
\nabla\cdot\Gamma^\alpha v
=\partial_r\Gamma^\alpha v\cdot\omega
+\frac1r\Omega\Gamma^\alpha v\cdot\omega^\perp,
\]
and so, by the divergence-free velocity constraint of \eqref{1}, we obtain \eqref{spconv}.
It also follows from \eqref{gradvident} that
\begin{equation}
\label{gradvident2}
r|\nabla \Gamma^\alpha v - \partial_r \Gamma^\alpha v \otimes\omega|\lesssim L_k,
\end{equation}
which will be used shortly in proving \eqref{sppde}.

Again by \eqref{der-decomp}, we may write for any matrix-valued function $H$
\begin{align}
\label{divid}
\partial_r H
&=\partial_r H \;I\\
\nonumber
&=\partial_r H[\omega\otimes\omega+\omega^\perp\otimes\omega^\perp]\\
\nonumber
&=\partial_r H \omega\otimes\omega+\partial_r H\omega^\perp\otimes\omega^\perp\\
\nonumber
&=[\nabla\cdot H-\frac1r\Omega H\omega^\perp]\otimes\omega
+[\nabla^\perp\cdot H+\frac1r\Omega H\omega]\otimes\omega^\perp.
\end{align}

Multiplying both sides of \eqref{divid} times the vector $\omega$,   we obtain
\[
\partial_r H\omega=\nabla\cdot H-\frac1r\Omega H\omega^\perp.
\]
Apply this to $H=\Gamma^\alpha G^\top$, and use
 the other divergence-free constraint from \eqref{1}.   We  thereby obtain
\begin{equation}
\label{gradgindent}
r|\partial_r \Gamma^\alpha G^\top\omega|
=|\Omega\Gamma^\alpha G^\top\omega^\perp|
\lesssim L_k,
\end{equation}
 which is \eqref{spcong}.

Next, apply \eqref{divid} to $H=\Gamma^\alpha G$ and use the constraint \eqref{1-1}:
\[
r\nabla\cdot\Gamma^\alpha G\otimes\omega-r\partial_r\Gamma^\alpha G
=\Omega \Gamma^\alpha G\omega^\perp\otimes\omega
-[rh_\alpha +\Omega \Gamma^\alpha G\omega]\otimes\omega^\perp.
\]
From this it follows that
\[
r|\nabla\cdot\Gamma^\alpha G-\partial_r\Gamma^\alpha G\omega|
=r|[\nabla\cdot\Gamma^\alpha G\otimes\omega-\partial_r\Gamma^\alpha G]\omega|
=|\Omega \Gamma^\alpha G\omega^\perp|
\lesssim L_k,
\]
proving \eqref{spdivg}, and also
\begin{equation}
\label{gradgindent2}
r|\nabla\cdot\Gamma^\alpha G\otimes\omega-\partial_r\Gamma^\alpha G|
\lesssim L_k+N_k.
\end{equation}
As an immediate consequence of this last inequality, we get
\[
r|\partial_r\Gamma^\alpha G\omega^\perp|
=r|[\nabla\cdot\Gamma^\alpha G\otimes\omega-\partial_r\Gamma^\alpha G]\omega^\perp|
\lesssim L_k+N_k,
\]
which proves \eqref{spcurlg}.
We are now ready to prove \eqref{sppde}.

Using   the PDEs \eqref{1} and the definition $S=t\partial_t+r\partial_r$, we can write
\begin{align*}
&t\nabla \Gamma^\alpha v+r\partial_r\Gamma^\alpha G=S\Gamma^\alpha G-tg_\alpha\\
&t\nabla\cdot\Gamma^\alpha G+r\partial_r\Gamma^\alpha v=S\Gamma^\alpha v -tf_\alpha+t\nabla\Gamma^\alpha p.
\end{align*}
This is rearranged as follows:
\begin{align*}
&t\nabla \Gamma^\alpha v+r\nabla\cdot\Gamma^\alpha G\otimes\omega
=r[\nabla\cdot\Gamma^\alpha G\otimes\omega-\partial_r\Gamma^\alpha G]+S\Gamma^\alpha G-tg_\alpha\\
&t\nabla\cdot\Gamma^\alpha G\otimes\omega+r\nabla\Gamma^\alpha v
=r[\nabla\Gamma^\alpha v-\partial_r\Gamma^\alpha v\otimes\omega]
+[S\Gamma^\alpha v -tf_\alpha+t\nabla\Gamma^\alpha p]\otimes\omega.
\end{align*}
Notice that by \eqref{gradvident2} and \eqref{gradgindent2}, the right-hand sides of these identities are  bounded by $L_k+N_k$.
Therefore, the bounds \eqref{sppde} result from the combination of these two identities.
\end{proof}

\begin{lem}
\label{divelltwo}
Let $(v,G)\in H^k_\Lambda(T)$.  Then for $|\alpha|\le k-2$, we have
\[
\|r\Gamma^\alpha v\cdot\omega\|_{L^\infty}+\|r\Gamma^\alpha G^\top\omega\|_{L^\infty}\lesssim E_{|\alpha|+2}^{1/2}.
\]
\end{lem}

\begin{proof}
We shall make use of the fact that
\begin{equation}
\label{omprop}
\Omega(v(x)\cdot\omega)=(\widetilde\Omega v(x))\cdot \omega,\quad
\widetilde\Omega(G(x)^\top\omega)=(\widetilde\Omega G(x))^\top \omega,\quad
\omega=x/r.
\end{equation}

By \eqref{gsob0} with $\lambda=2$ and \eqref{omprop}, we have
\begin{multline*}
\|r\Gamma^\alpha v\cdot\omega\|_{L^\infty}+\|r\Gamma^\alpha G^\top\omega\|_{L^\infty}\\
\lesssim \sum_{a=0,1}[\|r\partial_r\Omega^a(\Gamma^\alpha
v\cdot\omega)\|_{L^2}
+\|r\partial_r\widetilde\Omega^a( \Gamma^\alpha G^\top\omega)\|_{L^2}\\
+\|\Omega^a(\Gamma^\alpha v\cdot\omega)\|_{L^2}
+\|\widetilde\Omega^a(\Gamma^\alpha  G^\top\omega)\|_{L^2}]\\
=\sum_{a=0,1}[\|r(\partial_r\widetilde\Omega^a\Gamma^\alpha
v)\cdot\omega\|_{L^2}
+\|r\partial_r(\widetilde\Omega^a \Gamma^\alpha G^\top)\omega\|_{L^2}\\
+\|(\widetilde\Omega^a\Gamma^\alpha v)\cdot\omega\|_{L^2}
+\|(\widetilde\Omega^a\Gamma^\alpha  G^\top)\omega\|_{L^2}].\\
\end{multline*}
The result now follows by \eqref{spconv}, \eqref{spcong}.
\end{proof}

\section{Weighted $L^2$ Estimate}
\label{weight}

In this section, we show that the weighted norm is controlled by the energy, for small solutions.

\begin{lem}
\label{nonlinest}
Suppose that $(v,F)=(v,I+G)$,
$(v, G) \in H^k_\Lambda(T)$,  $k\ge4$, solves \eqref{Elas}, \eqref{derconstr}.
Then
\[
\|N_k(t)\|_{L^2}\lesssim E_k(t)+E_k(t)^{1/2}X_k(t)^{1/2}.
\]
\end{lem}

\begin{proof}

By \eqref{ndef} and Lemma \ref{lem4-1}, we have that
\begin{align*}
\|N_k(t)\|_{L^2}&\le\sum_{\tiny\begin{matrix}|\alpha|\le
k-1\end{matrix}}
[t\|\Gamma^\alpha p(t)\|_{L^2}+t\|f_\alpha(t)\|_{L^2}+t\|g_\alpha(t)\|_{L^2}+\|(t+r)h_\alpha(t)\|_{L^2}]\\
&\lesssim\sum_{\tiny\begin{matrix}|\alpha|\le
k-1\end{matrix}}[t\|f_\alpha(t)\|_{L^2}+t\|g_\alpha(t)\|_{L^2}+\|(t+r)h_\alpha(t)\|_{L^2}]
\end{align*}
In estimating these terms, we shall consider two regions: $r\le\jbtt$ and $r\ge \jbtt$.

\subsection*{ Estimates of nonlinearities for $ r\le\jbtt$.}

 Examining definitions \eqref{1}, \eqref{2-7}, we find that
\begin{multline*}
\sum_{\tiny\begin{matrix}|\alpha|\le k-1\end{matrix}}
[t\|f_\alpha\|_{L^2(r\le\jbtt)}+t\|g_\alpha\|_{L^2(r\le\jbtt)}+\|(t+r)h_\alpha\|_{L^2(r\le\jbtt)}]\\
\lesssim\sum_{\tiny\begin{matrix}\beta+\gamma=\alpha\\|\alpha|\le
k-1\end{matrix}} \jbt\|(|\Gamma^\beta v|+|\Gamma^\beta G|)
(|\nabla\Gamma^\gamma v|+|\nabla\Gamma^\gamma
G|)\|_{L^2(r\le\jbtt)}.
\end{multline*}
To simplify the notation a bit, we shall write
\[
|(\Gamma^kv,\Gamma^kG)|=\sum_{|\alpha|\le k}[|\Gamma^\alpha v|+|\Gamma^\alpha G|],
\]
(and similar).
We make use of the fact that since  $\beta+\gamma=\alpha$,  $|\alpha|\le k-1$, and $k\ge4$,
 either $|\beta|+2\le k$ or $|\gamma|+3\le k$.
Thus, we have by \eqref{gsob1}
\begin{align*}
\jbt\|(|\Gamma^\beta v|&+|\Gamma^\beta G|)
(|\nabla\Gamma^\gamma v|+|\nabla\Gamma^\gamma G|)\|_{L^2(r\le\jbtt)}\\
\lesssim &\jbt
\|(\Gamma^{k-2} v,\Gamma^{k-2}G)\|_{L^\infty}
\|(\nabla\Gamma^{k-1} v,\nabla\Gamma^{k-1} G)\|_{L^2(r\le\jbtt)}\\
&+\jbt
\|(\Gamma^{k} v,\Gamma^{k}G)\|_{L^2}
\|(\nabla\Gamma^{k-3} v,\nabla\Gamma^{k-3} G)\|_{L^\infty(r\le\jbtt)}\\
\lesssim &
\|(\Gamma^{k-2} v,\Gamma^{k-2}G)\|_{L^\infty}
\|\jbtr(\nabla\Gamma^{k-1} v,\nabla\Gamma^{k-1} G)\|_{L^2(r\le\jbtt)}\\
&+
\|(\Gamma^{k} v,\Gamma^{k}G)\|_{L^2}
\sum_{|\lambda|\le2}
\|\jbtr\partial^\lambda(\nabla\Gamma^{k-3} v,\nabla\Gamma^{k-3} G)\|_{L^\infty(r\le\jbtt)}\\
\lesssim &E_k(t)^{1/2}X_k(t)^{1/2}.
\end{align*}
\subsection*{ Estimates of nonlinearities for $ r\ge\jbtt$.}
Using \eqref{der-decomp}, we replace all  derivatives which occur in \eqref{2-5}, \eqref{2-7}
by their radial and angular parts.
 We find that
\begin{multline*}
|f_\alpha|+|g_\alpha|+|h_\alpha| \lesssim
\sum_{\tiny\begin{matrix}\beta+\gamma=\alpha\\|\alpha|\le
k-1\end{matrix}}
(|\Gamma^\beta v\cdot\omega|+|\Gamma^\beta G^\top\omega|)(|\partial_r\Gamma^\gamma v|+|\partial_r\Gamma^\gamma G|)\\
+\sum_{\tiny\begin{matrix}\beta+\gamma=\alpha\\|\alpha|\le
k-1\end{matrix}} (|\Gamma^\beta v|+|\Gamma^\beta
G|)\frac1r(|\Omega\Gamma^\gamma v|+|\Omega\Gamma^\gamma G|).
\end{multline*}
Thus, we obtain
\begin{multline}\label{mess}
\sum_{|\alpha|\le k-1}[t\|f_\alpha\|_{L^2(r\ge\jbtt)}+t\|g_\alpha\|_{L^2(r\ge\jbtt)}+\|(t+r)h_\alpha\|_{L^2(r\ge\jbtt)}]\\
\lesssim \sum_{\tiny\begin{matrix}\beta+\gamma=\alpha\\|\alpha|\le
k-1\end{matrix}}
\|r(|\Gamma^\beta v\cdot\omega|+|\Gamma^\beta G^\top\omega|)(|\partial_r\Gamma^\gamma v|+|\partial_r\Gamma^\gamma G|)\|_{L^2(r\ge\jbtt)}\\
+\sum_{\tiny\begin{matrix}\beta+\gamma=\alpha\\|\alpha|\le
k-1\end{matrix}} \|(|\Gamma^\beta v|+|\Gamma^\beta
G|)(|\Omega\Gamma^\gamma v|+|\Omega\Gamma^\gamma
G|)\|_{L^2(r\ge\jbtt)}.
\end{multline}
We claim that all terms on the right-hand side of \eqref{mess} can be bounded by $E_k$.

  Consider  the term in the first sum which has
$\beta=\alpha$, $\gamma=0$.  It can be estimated as follows:
\begin{align}
\nonumber
\|r(|\Gamma^\alpha v\cdot\omega|+&|\Gamma^\alpha G^\top\omega|)(|\partial_r   v|+|\partial_r G|)\|_{L^2(r\ge\jbtt)}\\
\label{step}
&\lesssim  \Bigg(\sup_{r\ge\jbtt}r^2\int_{\mathbb S^1}
(|\Gamma^\alpha v(r\omega)\cdot\omega|^2+|\Gamma^\alpha G^\top(r\omega)\omega|^2)d\sigma\Bigg)^{1/2}\\
\nonumber
&\qquad\qquad\times \Bigg(\int_{\jbtt}^\infty \sup_{\mathbb S^1}(|\partial_r  v(r\omega)|^2+|\partial_r  G(r\omega)|^2)rdr\Bigg)^{1/2}.
\end{align}
Using  \eqref{rgsob0} with $\lambda=2$ from Lemma \ref{rgsob},
the first of these integrals is bounded by
\[
\|r\partial_r\Gamma^\alpha
v\cdot\omega\|_{L^2}+\|r\partial_r\Gamma^\alpha G^\top\omega\|_{L^2}
+\|\Gamma^\alpha v\|_{L^2}+\|\Gamma^\alpha G\|_{L^2}.
\]
Noticing the form of the first two terms, we can use \eqref{spconv},\eqref{spcong} to bound this by $E_{|\alpha|+1}^{1/2}\le E_k^{1/2}$.

The second integral in \eqref{step} is bounded as follows using Lemma
\ref{lem3-4}:
\[
\sum_{a=0,1}(\|\Omega^a\partial_r   v\|_{L^2}+\|\Omega^a\partial_r  G\|_{L^2})\lesssim E_2^{1/2}.
\]
This proves that the term in \eqref{step} is bounded by $E_k$, as claimed.

For the remaining terms in the first sum  in \eqref{mess}, we have  $\beta\ne\alpha$.  We estimate these as follows:
\begin{multline}
\label{fragment}
\|r(|\Gamma^\beta v\cdot\omega|+
|\Gamma^\beta G^\top\omega|)(|\partial_r\Gamma^\gamma v|+|\partial_r\Gamma^\gamma G|)\|_{L^2(r\ge\jbtt)}\\
\lesssim
(\|r\Gamma^\beta v\cdot\omega\|_{L^\infty}+\|r\Gamma^\beta G^\top\omega\|_{L^\infty})
(\|\partial_r\Gamma^\gamma v\|_{L^2}+\|\partial_r\Gamma^\gamma G\|_{L^2}).
\end{multline}
The terms in $L^\infty$ are estimated
by $E_{|\beta|+2}^{1/2}$, using Lemma \ref{divelltwo}, and
thus, since $|\beta|\le|\alpha|-1\le k-2$, the expression \eqref{fragment} is bounded by
\[
E_{|\beta|+2}^{1/2}E_{|\gamma|+1}^{1/2}
\lesssim E_k.
\]
Altogether, the first sum of terms in \eqref{mess} is bounded by $E_k$.

The second group of terms on the right of \eqref{mess}
is also bounded by $E_k$, using the same interpolation strategy as in the case $r\le\jbtt$.

\end{proof}

\begin{lem}
\label{div2grad}
Suppose that
$(v, G) \in H^k_\Lambda(T)$,  solves \eqref{Elas}, \eqref{derconstr}.
\[
\|(t-r)\nabla\Gamma^\alpha G\|_{L^2}\lesssim E_k(t)^{1/2}+ \|(t-r)\nabla\cdot\Gamma^\alpha G\|_{L^2}
+\|(t+r)h_\alpha\|_{L^2}.
\]
\end{lem}

\begin{proof}
For any $\rr^2\otimes\rr^2$-valued function $H$, we have that
\[
|\nabla H|^2-[|\nabla\cdot H|^2+|\nabla^\perp\cdot H|^2]=
-2[\partial_1(H_{i1}\partial_2H_{i2})-\partial_2(H_{i1}\partial_1H_{i2})].
\]
Thus, using integration by parts and  Young's inequality, we get
\begin{align*}
\|(t-r)\nabla H\|_{L^2}^2&-\left[\|(t-r)\nabla\cdot  H\|_{L^2}^2+\|(t-r)\nabla^\perp\cdot H\|_{L^2}^2\right]\\
&=-\int 2[\partial_1(H_{i1}\partial_2H_{i2})-\partial_2(H_{i1}\partial_1H_{i2})]dx\\
&=-\int4(t-r)[\omega_1(H_{i1}\partial_2H_{i2})-\omega_2(H_{i1}\partial_1H_{i2})]dx\\
&\le\frac12\|(t-r)\nabla H\|_{L^2}^2+C\|H\|_{L^2}^2,
\end{align*}
and so we obtain
\[
\|(t-r)\nabla H\|_{L^2}^2\lesssim \|(t-r)\nabla\cdot  H\|_{L^2}^2+\|(t-r)\nabla^\perp\cdot H\|_{L^2}^2 + \|H\|_{L^2}^2.
\]
The Lemma follows by applying this inequality to $H=\Gamma^\alpha G$ and then using \eqref{1-1}.

\end{proof}

\begin{thm}
\label{weightedelltwo}
Suppose that $(v,F)=(v,I+G)$,
$(v, G) \in H^k_\Lambda(T)$,  $k\ge4$, solves \eqref{Elas}, \eqref{derconstr}.
If $E_k(t)\ll1$, then $X_k(t)\lesssim E_k(t)^{1/2}$.
\end{thm}

\begin{proof}

Starting with  definition \eqref{WEnergy} and using the fact that $\jbtr\le 1+|t-r|$, we obtain from
 Lemma \ref{div2grad}
\begin{align*}
X_k^2&=\sum_{|\alpha|\le k-1}[\|\jbtr\nabla\Gamma^\alpha v\|_{L^2}^2+\|\jbtr\nabla\Gamma^\alpha G\|_{L^2}^2]\\
&\lesssim E_k+\sum_{|\alpha|\le k-1}[\|(t-r)\nabla\Gamma^\alpha v\|_{L^2}^2+\|(t-r)\nabla\Gamma^\alpha G\|_{L^2}^2]\\
&\lesssim E_k+\sum_{|\alpha|\le k-1}[\|(t-r)\nabla\Gamma^\alpha v\|_{L^2}^2+\|(t-r)\nabla\cdot\Gamma^\alpha G\|_{L^2}^2+\|N_k\|_{L^2}^2].\\
\end{align*}
Since
\[
\nabla\Gamma^\alpha v=\frac12[\nabla\Gamma^\alpha
v+\nabla\cdot\Gamma^\alpha G\otimes\omega]
+\frac12[\nabla\Gamma^\alpha v-\nabla\cdot\Gamma^\alpha
G\otimes\omega]
\]
and
\[
\nabla\cdot\Gamma^\alpha G=\frac12[\nabla\Gamma^\alpha v+\nabla\cdot\Gamma^\alpha G\otimes\omega]\omega
-\frac12[\nabla\Gamma^\alpha v-\nabla\cdot\Gamma^\alpha G\otimes\omega]\omega,
\]
we see that
\begin{multline*}
|t-r|\;[|\nabla\Gamma^\alpha v|+|\nabla\cdot\Gamma^\alpha G|]\\
\lesssim |t+r|\;|\nabla\Gamma^\alpha v+\nabla\cdot\Gamma^\alpha G\otimes\omega|
+|t-r|\;|\nabla\Gamma^\alpha v-\nabla\cdot\Gamma^\alpha G\otimes\omega|.
\end{multline*}
It follows from \eqref{sppde} that
\[
\|(t-r)\nabla\Gamma^\alpha v\|_{L^2}^2+\|(t-r)\nabla\cdot\Gamma^\alpha G\|_{L^2}^2
\lesssim E_k+\|N_k\|^2.
\]
and thus,
 we obtain
\[
X_k\lesssim E_k^{1/2}+\|N_k\|_{L^2}.
\]
Then applying Lemma \ref{nonlinest}, we get
\[
X_k\lesssim E_k^{1/2}+E_k+E_k^{1/2}X_k^{1/2}.
\]
Under the assumption that $E_k\ll1$, we obtain
\[
X_k\lesssim E_k^{1/2}+E_k\lesssim E_k^{1/2}.
\]
\end{proof}

\section{Estimates for Special Quantities, II}
\label{special2}
With the results of the previous section in hand, we can now complete
the estimation of $\Gamma^\alpha v+\Gamma^\alpha G\omega$ and $\Gamma^\alpha G\omega^\perp$.
\begin{lem}
\label{goodder}
Let $k \geq 4$, $E_k\ll1$, $\omega = {x}/{|x|}$. Then we have
\begin{align}
\label{goodderelltwo}
&\|r(\partial_r\Gamma^\alpha v+\partial_r\Gamma^\alpha G\omega)\|_{L^2}
+\|r\partial_r\Gamma^\alpha G\omega^\perp\|_{L^2}\lesssim E_k^{1/2},&& |\alpha|\le k-1,\\
\label{goodderellinf}
&\|r(\Gamma^\alpha v+\Gamma^\alpha G\omega)\|_{L^\infty}
+\|r\Gamma^\alpha G\omega^\perp\|_{L^\infty}\lesssim E_k^{1/2},&& |\alpha|\le k-2.
\end{align}
\end{lem}

\begin{proof}
First, note that
\[
\partial_r\Gamma^\alpha v+\partial_r\Gamma^\alpha G\omega
=(\nabla\Gamma^\alpha  v+\nabla\cdot\Gamma^\alpha G\otimes\omega)\omega
+(\partial_r \Gamma^\alpha G\omega-\nabla\cdot\Gamma^\alpha G).
\]
Therefore, by \eqref{spdivg} and \eqref{sppde}, we have
\[
r|\partial_r\Gamma^\alpha v+\partial_r\Gamma^\alpha G\omega|
\le (r+t)|\nabla\Gamma^\alpha  v+\nabla\cdot\Gamma^\alpha G\otimes\omega|
+r|\partial_r \Gamma^\alpha G\omega-\nabla\cdot\Gamma^\alpha G|
\lesssim L_k+N_k.
\]
Combining this with \eqref{spcurlg}, we get the estimate
\[
\|r(\partial_r\Gamma^\alpha v+\partial_r\Gamma^\alpha G\omega)\|_{L^2}
+\|r\partial_r\Gamma^\alpha G\omega^\perp\|_{L^2}
\le E_k^{1/2}+\|N_k\|_{L^2}.
\]
Estimate \eqref{goodderelltwo} now follows by Lemma \ref{nonlinest} and Theorem \ref{weightedelltwo}.

To obtain the other estimate, we observe that (similar to \eqref{omprop})
\begin{equation}
\label{omprop-1}
\widetilde\Omega(\Gamma^\alpha G\omega)=(\widetilde\Omega \Gamma^\alpha G)\omega,
\qquad
\widetilde\Omega(\Gamma^\alpha G\omega^\perp)=(\widetilde\Omega \Gamma^\alpha G)\omega^\perp.
\end{equation}
By \eqref{gsob0} with $\lambda=2$, \eqref{omprop-1}, \eqref{goodderelltwo}, we have
\begin{multline*}
\|r(\Gamma^\alpha v+\Gamma^\alpha G\omega)\|_{L^\infty}
+\|r\Gamma^\alpha G\omega^\perp\|_{L^\infty}\\
\lesssim
\sum_{a=0,1}[\|r(\partial_r\widetilde\Omega^a\Gamma^\alpha v+\partial_r\widetilde\Omega^a\Gamma^\alpha G\omega)\|_{L^2}
+\|\widetilde\Omega^a\Gamma^\alpha v+\widetilde\Omega^a\Gamma^\alpha G\omega\|_{L^2}\\
+\|r\partial_r\widetilde\Omega^a\Gamma^\alpha G\omega^\perp\|_{L^2}
+\|\widetilde\Omega^a\Gamma^\alpha G\omega^\perp\|_{L^2}]
\lesssim E_k^{1/2}.
\end{multline*}

\end{proof}

\section{Energy Estimate with a Ghost Weight}
\label{ghostenergy}

Choose $q = q(t - r)$, with $q(\sigma) = \int_0^\sigma\la z\ra^{-2} dz$
so that $q^\prime(\sigma) = \la \sigma\ra^{-2}$
and
$|q(\sigma)| \leq \frac{\pi}{2}$. Let $|\alpha| \leq k$. Taking
the inner product of the first and second equation in \eqref{1}
with $e^{- q}\Gamma^\alpha v$ and $e^{- q}\Gamma^\alpha G$
respectively, and then adding them up, we find
\begin{multline*}
\int\Big(e^{- q}\partial_t(|\Gamma^\alpha v|^2 + |\Gamma^\alpha
  G|^2) - 2e^{- q}(\Gamma^\alpha v_i\partial_j\Gamma^\alpha G_{ij}
  + \Gamma^\alpha G_{ij}\partial_j\Gamma^\alpha v_i)\Big)dx\\
= - 2\int e^{- q}\Gamma^\alpha v\cdot\nabla\Gamma^\alpha pdx +
  2\int e^{- q}(f_\alpha\cdot \Gamma^\alpha v + (g_\alpha)_{ij} \Gamma^\alpha G_{ij})dx.
\end{multline*}
Integration by parts gives that
\begin{align}\label{8-1}
\frac{d}{dt}\int e^{- q}&(|\Gamma^\alpha v|^2 + |\Gamma^\alpha
  G|^2)dx\\\nonumber
&= - \int e^{- q}\big[\partial_tq(|\Gamma^\alpha v|^2 +
  |\Gamma^\alpha G|^2) - 2\partial_jq\Gamma^\alpha
  v_i\Gamma^\alpha G_{ij}\big]dx\\\nonumber
&\qquad -\ 2\int e^{- q}\Gamma^\alpha v\cdot\nabla\Gamma^\alpha
  pdx + 2\int e^{- q}(f_\alpha\cdot \Gamma^\alpha v + (g_\alpha)_{ij} \Gamma^\alpha
  G_{ij})dx\\\nonumber
&= - \int \frac{e^{- q}}{\la t - r\ra^2}\big(|\Gamma^\alpha v +
  \Gamma^\alpha G\omega|^2 +
  |\Gamma^\alpha G\omega^\perp|^2\big)dx\\\nonumber
&\qquad -\ 2\int e^{- q}\Gamma^\alpha v\cdot\nabla\Gamma^\alpha
  pdx + 2\int e^{-q}\big[f_\alpha\cdot \Gamma^\alpha v
  + (g_\alpha)_{ij} \Gamma^\alpha G_{ij}\big]dx.
\end{align}
We emphasize that here we do not use integration by parts in the
term involving pressure.
We also point out that we cannot use the approach of Lemma \ref{nonlinest} to estimate the
nonlinear terms because
we now have that $|\alpha|\le k$ rather than $|\alpha|\le k-1$ as we had earlier.

Let us first treat the last term in \eqref{8-1}. Recall that
$f_\alpha$ and $g_\alpha$ are given by \eqref{2-5}. Since
 $v$ and $G^\top$ are divergence-free, we get
 \begin{align}\label{8-2}
\int e^{-q}&\big[f_\alpha\cdot \Gamma^\alpha v + (g_\alpha)_{ij}
  \Gamma^\alpha G_{ij}\big]dx\\\nonumber
&= \sum_{\tiny\begin{matrix}\beta + \gamma = \alpha\\ \gamma \neq \alpha\end{matrix}}
  \int e^{-q}\Gamma^\alpha v_i
  \big[\partial_j\Gamma^\gamma G_{ik} \Gamma^\beta G_{jk} -
  \Gamma^\beta v_j\partial_j\Gamma^\gamma v_i\big]dx\\\nonumber
&\qquad +\ \sum_{\tiny\begin{matrix}\beta + \gamma = \alpha\\ \gamma \neq \alpha\end{matrix}}
\int  e^{-q}\Gamma^\alpha G_{ik}\big[\partial_j\Gamma^\gamma v_{i} \Gamma^\beta G_{jk}
  - \Gamma^\beta v_j\partial_j\Gamma^\gamma
  G_{ik}\big]dx\\\nonumber
&\qquad +\ \frac{1}{2}\int e^{-q}\partial_j
  \big[2\Gamma^\alpha v_i\Gamma^\alpha G_{ik} G_{jk} -
  v_j(|\Gamma^\alpha v|^2 + |\Gamma^\alpha G|^2)\big]dx.
\end{align}
To estimate the last term in \eqref{8-2}, we first
compute that
\begin{multline*}
\int e^{-q}\partial_j
  \big[2\Gamma^\alpha v_i\Gamma^\alpha G_{ik} G_{jk} -
  v_j(|\Gamma^\alpha v|^2 + |\Gamma^\alpha G|^2)\big]dx\\
= - \int \jbtr^{-2}e^{-q}\big[2\Gamma^\alpha
  v_i\Gamma^\alpha G_{ik} G_{jk}\omega_j -
  (|\Gamma^\alpha v|^2 + |\Gamma^\alpha
  G|^2)v_j\omega_j\big]dx.
\end{multline*}
This is estimated by
\[
(\|\jbtr^{-2}v\cdot\omega\|_{L^\infty}+\|\jbtr^{-2}G^\top\omega\|_{L^\infty})E_k(t).
\]
Now by Lemma \ref{divelltwo}, we have
\begin{align*}
\|\jbtr^{-2}v\cdot\omega\|_{L^\infty}
&\le\|\jbtr^{-2}v\cdot\omega\|_{L^\infty(r\le\jbtt)}+\|\jbtr^{-2}v\cdot\omega\|_{L^\infty(r\ge\jbtt)}\\
&\lesssim \jbt^{-2}\|v\|_{L^\infty(r\le\jbtt)}+\|v\cdot\omega\|_{L^\infty(r\ge\jbtt)}\\
&\lesssim \jbt^{-2}\|v\|_{L^\infty(r\le\jbtt)}+\jbt^{-1}\|rv\cdot\omega\|_{L^\infty(r\ge\jbtt)}\\
&\lesssim \jbt^{-1}E_k^{1/2}.
\end{align*}
A similar estimate holds for the term with $G^\top\omega$.  We have shown that the last term in \eqref{8-2}
is bounded by $\jbt^{-1}E_k^{3/2}$.

Next, we are going to estimate the first and second terms on the
right hand side of \eqref{8-2}.  Since $k\ge5$, it is always the case that $|\gamma|\le k-1$ and
either  $|\beta|\le k-2$ or $|\gamma|\le k-3$.

For $r < \jbtt$, we use \eqref{standard-sob} and
Theorem \ref{weightedelltwo} to get
\begin{align}\label{8-4}
\sum_{\tiny\begin{matrix}\beta + \gamma = \alpha\\ \gamma \neq \alpha\end{matrix}}
  &\Big|\int_{r < \jbtt} e^{-q}\Gamma^\alpha v_i
  \big[\partial_j\Gamma^\gamma G_{ik} \Gamma^\beta G_{jk} -
  \Gamma^\beta v_j\partial_j\Gamma^\gamma v_i\big]dx\\ \nonumber
  +&\int_{r <\jbtt}e^{-q}
\Gamma^\alpha G_{ik}\big[\partial_j\Gamma^\gamma v_{i}
  \Gamma^\beta G_{jk} - \Gamma^\beta v_j\partial_j\Gamma^\gamma
  G_{ik}\big]dx\Big|\\\nonumber
&\lesssim \jbt^{-1}E_{k}(t)^{1/2}
  \Big(\|\jbtr(\nabla\Gamma^{k - 1}v, \nabla\Gamma^{k - 1}G)\|_{L^2(r < \jbtt)}\|(\Gamma^{k - 2}v, \Gamma^{k - 2}G)\|_{L^\infty}\\\nonumber
&\qquad +\ \|\jbt(\nabla\Gamma^{k - 3}v, \nabla\Gamma^{k - 3}G)\|_{L^\infty(r <\jbtt)}\|\Gamma^{k}(v, G)\|_{L^2}\Big)\\\nonumber
&\lesssim {\la t\ra}^{-1}E_{k}(t)^{1/2}
  \Big(X_{k}(t)^{1/2}E_{k }(t)^{1/2} \Big)\\\nonumber
&\lesssim {\la t\ra}^{-1}E_{k}(t)^{3/2}.
\end{align}

For $r>\jbtt$, we write using \eqref{der-decomp}
{\allowdisplaybreaks
\begin{align}
\label{main-nonlin}
\sum_{\tiny\begin{matrix}\beta + \gamma = \alpha\\ \gamma \neq \alpha\end{matrix}}
  &\Big|\int_{r >\jbtt} e^{-q}\Gamma^\alpha v_i
  \big[\partial_j\Gamma^\gamma G_{ik} \Gamma^\beta G_{jk} -
  \Gamma^\beta v_j\partial_j\Gamma^\gamma v_i\big]dx\\ \nonumber
 & +\int_{r >\jbtt}e^{-q}
\Gamma^\alpha G_{ik}\big[\partial_j\Gamma^\gamma v_{i}
  \Gamma^\beta G_{jk} - \Gamma^\beta v_j\partial_j\Gamma^\gamma
  G_{ik}\big]dx\Big|\\\nonumber
  \le&\sum_{\tiny\begin{matrix}\beta + \gamma = \alpha\\ \gamma \neq \alpha\end{matrix}}
  \Big|\int_{r > \jbtt} e^{-q}\Gamma^\alpha v_i
  \big[\omega_j\partial_r\Gamma^\gamma G_{ik} \Gamma^\beta G_{jk} -
  \Gamma^\beta v_j\omega_j\partial_r\Gamma^\gamma v_i\big]dx\\ \nonumber
  &+\int_{r >\jbtt}e^{-q}
\Gamma^\alpha G_{ik}\big[\omega_j\partial_r\Gamma^\gamma v_{i}
  \Gamma^\beta G_{jk} - \Gamma^\beta v_j\omega_j\partial_r\Gamma^\gamma
  G_{ik}\big]dx\Big|\\ \nonumber
  &+  \int_{r>\jbtt}R_\alpha dx\\\nonumber
 = &\sum_{\tiny\begin{matrix}\beta + \gamma = \alpha\\ \gamma \neq \alpha\end{matrix}}
  \Big|\int_{r > \jbtt} e^{-q}\big\la(\Gamma^\alpha v,\Gamma^\alpha G),
  B[(\Gamma^\beta v,\Gamma^\beta G),(\partial_r\Gamma^\gamma v,\partial_r\Gamma^\gamma G)]\big\ra dx
  \Big|\\ \nonumber
  & + \int_{r>\jbtt}R_\alpha dx,
\end{align}
}
in which
\[
B[(v_1,G_1),(v_2,G_2)]=\big(G_2G_1^\top\omega-(v_1\cdot\omega)v_2,v_2\otimes G_1^\top\omega-(v_1\cdot\omega)G_2\big),
\]
and
\[
|R_\alpha|\lesssim\frac1r\sum_{\tiny\begin{matrix}\beta + \gamma = \alpha\\ \gamma \neq \alpha\end{matrix}}
|((\Gamma^\alpha v,\Gamma^\alpha G)|\;
|((\Gamma^\beta v,\Gamma^\beta G)|\;
|(\widetilde\Omega\Gamma^\gamma v,\widetilde\Omega\Gamma^\gamma G)|.
\]
To analyze the structure of these terms when $r > \jbtt$, we decompose vector/matrix pairs $(v,G)$ as
\begin{align*}
(v,G)&=\sum_{k=-1}^1\Pi_k(v,G)\\
\Pi_1(v,G)&= {\textstyle\frac12}((v+G\omega),(v+G\omega)\otimes\omega))\\
\Pi_{-1}(v,G)&= {\textstyle\frac12}((v-G\omega),-(v-G\omega)\otimes\omega))\\
\Pi_0(v,G)&=(0,G\omega^\perp\otimes\omega^\perp).
\end{align*}
We have the following cancellations
\begin{align*}
&B[\Pi_k(v_1,G_1),\Pi_k(v_2,G_2)]=0,&&k=\pm1\\
&B[\Pi_0(v_1,G_1),\Pi_k(v_2,G_2)]=0,&&k=\pm1\\
&B[\Pi_k(v_1,G_1),\Pi_0(v_2,G_2)]=0,&&k=\pm1.
\end{align*}
In particular, there is no self-interaction of the ``bad" quantity $\Pi_{-1}$.  Thus, we have
\begin{multline*}
B[(v_1,G_1),(v_2,G_2)]=B[\Pi_1(v_1,G_1),\Pi_{-1}(v_2,G_2)]\\
+B[\Pi_{-1}(v_1,G_1),\Pi_1(v_2,G_2)]+B[\Pi_0(v_1,G_1),\Pi_0(v_2,G_2)].
\end{multline*}
The conclusion of this discussion is that the terms in \eqref{main-nonlin} are bounded by
{\allowdisplaybreaks
\begin{align*}
\sum_{\tiny\begin{matrix}\beta + \gamma = \alpha\\ \gamma \neq \alpha\end{matrix}}
&\int_{r>\jbtt} |(\Gamma^\alpha v,\Gamma^\alpha G)|\;
\Big[
|\Gamma^\beta v+\Gamma^\beta G\omega|\;|(\partial_r\Gamma^\gamma v,\partial_r\Gamma^\gamma G)|\\
&+|(\Gamma^\beta v,\Gamma^\beta G)|\;|\partial_r\Gamma^\gamma v+\partial_r\Gamma^\gamma G\omega|\\
&+|(\Gamma^\beta v,\Gamma^\beta G)|\;|\partial_r\Gamma^\gamma G\omega^\perp|\\
&+\frac1r|((\Gamma^\beta v,\Gamma^\beta G)|\;
|(\widetilde\Omega\Gamma^\gamma v,\widetilde\Omega\Gamma^\gamma G)|\Big]dx\\
\le& C \|(\Gamma^k v,\Gamma^k G)\|_{L^2}\\
&\times\Big[\jbt^{-1}\|r(\Gamma^{k-2} v+\Gamma^{k-2} G\omega)\|_{L^\infty(r>\jbtt)}
\|(\partial_r\Gamma^{k-1} v,\partial_r\Gamma^{k-1} G)\|_{L^2}\\
&+\jbt^{-1/2}\|\frac{\Gamma^{k} v+\Gamma^{k} G\omega}{\jbtr}\|_{L^2}
\|r^{1/2}\jbtr(\partial_r\Gamma^{k-3} v,\partial_r\Gamma^{k-3} G)\|_{L^\infty(r>\jbtt)}\\
&+\jbt^{-1}\|(\Gamma^{k-2}v,\Gamma^{k-2} G)\|_{L^\infty}
\|r|\partial_r\Gamma^{k-1} v+\partial_r\Gamma^{k-1} G\omega|
+r|\partial_r\Gamma^{k-1} G\omega^\perp|\|_{L^2(r>\jbtt)}\\
&+\jbt^{-1}\|(\Gamma^k v,\Gamma^k G)\|_{L^2}
\|r|\partial_r\Gamma^{k-3} v+\partial_r\Gamma^{k-3} G\omega|
+r|\partial_r\Gamma^{k-3} G\omega^\perp|\|_{L^\infty(r>\jbtt)}\\
&+\jbt^{-1}E_k\Big].
\end{align*}
}
By \eqref{goodderellinf},  we have that
\[
\|r(\Gamma^{k-2} v+\Gamma^{k-2} G\omega)\|_{L^\infty(r>\jbtt)}
\lesssim E_k^{1/2}.
\]
By \eqref{gsob1}, we obtain
\begin{multline*}
\jbt^{-1/2}\left\|\frac{\Gamma^{k} v+\Gamma^{k} G\omega}{\jbtr}\right\|_{L^2}
\|r^{1/2}\jbtr(\partial_r\Gamma^{k-3} v,\partial_r\Gamma^{k-3} G)\|_{L^\infty(r>\jbtt)}\\
\le \jbt^{-1/2}\left\|\frac{\Gamma^{k} v+\Gamma^{k}G\omega}{\jbtr}\right\|_{L^2}
 E_k^{1/2}.
\end{multline*}
By the conventional Sobolev imbedding, \eqref{goodderelltwo}, and \eqref{goodderellinf} we have
\[
\|(\Gamma^{k-2}v,\Gamma^{k-2} G)\|_{L^\infty}
\|r|\partial_r\Gamma^{k-1} v+\partial_r\Gamma^{k-1} G\omega|
+r|\partial_r\Gamma^{k-1} G\omega^\perp|\|_{L^2}\lesssim E_k,
\]
and
\[
\|(\Gamma^k v,\Gamma^k G)\|_{L^2}
\|r|\partial_r\Gamma^{k-3} v+\partial_r\Gamma^{k-3} G\omega|
+r|\partial_r\Gamma^{k-3} G\omega^\perp|\|_{L^\infty}\lesssim E_k.
\]

It follows, therefore, that the integral \eqref{8-2} is bounded by
\[
 \mu\left\|\frac{\Gamma^{k} v+\Gamma^{k}G\omega}{\jbtr}\right\|_{L^2}^2
+C_\mu\jbt^{-1} E_k^{3/2},
\]
for an arbitrarily small $\mu>0$.

It remains to treat the pressure term in \eqref{8-1}.
However, thanks to \eqref{pressure-2}, this term is handled exactly as the preceding ones.

Finally, we gather our estimates for  \eqref{8-1} to
get
\[
\widetilde E_k'(t)+\left\|\frac{\Gamma^{k} v+\Gamma^{k}G\omega}{\jbtr}\right\|_{L^2}^2
\le
\mu\left\|\frac{\Gamma^{k} v+\Gamma^{k}G\omega}{\jbtr}\right\|_{L^2}^2
+C_\mu \jbt^{-1} E_k^{3/2},
\]
with
\[
\widetilde E_k(t)=\sum_{|\alpha|\le k}\int e^{- q}(|\Gamma^\alpha v|^2 + |\Gamma^\alpha
  G|^2)dx.
\]
Noting that $E_k\sim \widetilde E_k$, we obtain, for $\mu$ small,
\[
\widetilde E_k'(t)\le C_0\jbt^{-1}\widetilde E_k(t)^{3/2}, \quad 0\le t<T.
\]
This implies that $E_k(t)$ remains bounded by $2\epsilon^2$ on a time interval
of order $T\sim \exp (C_0/\epsilon)$.

\section{General Isotropic Elastodynamics}
\label{general}

For general isotropic elastodynamics, the energy functional has the form $W = W(F)$ with
\begin{equation}\label{9-5}
W(F) = W(QF)= W(FQ)
\end{equation}
for all rotation  matrices: $Q=Q^\top$, $\det Q = 1$. The first relation is
due to  frame indifference, while the second one expresses the
isotropy of materials.
This implies that $W$ depends on $F$ through the principal invariants of $FF^\top$, namely $\tr FF^\top$ and $\det FF^\top$ in 2-D.
Setting  $\tau=\frac12\;\tr FF^\top$ and $\delta=\det F=(\det FF^\top)^{1/2}$,  we may assume that $W(F)=\bar W(\tau,\delta)$,
for some smooth function $\bar W:\rr^+\times\rr^+\to\rr^+$.

Since
\[
\frac{\partial\tau}{\partial F}= F\quad\mbox{and}\quad \frac{\partial\delta}{\partial F}=\delta F^{-T},
\]
the Piola-Kirchhoff
stress    has the form
\begin{equation}\nonumber
S(F) \equiv\frac{\partial W(F)}{\partial F}=\bar W_\tau(\tau,\delta)\;F+\bar W_\delta(\tau,\delta)\delta F^{-T}.
\end{equation}
We  assume that the reference configuration is
stress free, $S(I)=0$, so that
\begin{equation}
\label{stressfree}
\bar W_\tau(1,1)+\bar W_\delta(1,1)=0.
\end{equation}
 The Cauchy stress tensor is
 \[
 T(F)\equiv \delta^{-1} S(F)F^\top=\delta^{-1}\bar W_\tau(\tau,\delta)FF^\top+\bar W_\delta(\tau,\delta) I,
 \]
 and the term $\nabla\cdot FF^\top$ in
\eqref{Elas} is replaced by $\nabla\cdot T(F)$.  Let us now proceed to
examine this term.

Write
\begin{align*}
T(F)= &\bar W_\tau(1,1)FF^\top\\
&+[\delta^{-1}\bar W_\tau(\tau,\delta)-\bar W_\tau(1,1)][FF^\top-I]\\
&+[\delta^{-1}\bar W_\tau(\tau,\delta)-\bar W_\tau(1,1)+\bar W_\delta(\tau,\delta)]I\\
\equiv& \sum_{a=1}^3T_a(F).
\end{align*}

Assume that
\begin{equation}
\label{ellipticity}
\bar W_\tau(1,1)>0.
\end{equation}
 Then $T_1(F)$ gives rise to a Hookean term.
Notice that  assumption \eqref{ellipticity} rules out the hydrodynamical case $\bar W_\tau=0$.
 The principal invariants can be expanded about the identity as follows:
\begin{equation}
\label{tauexp}
\tau=\frac12\tr FF^\top=\frac12\tr (I+G)(I+G^\top)=1+\tr G +\frac12\tr GG^\top
\end{equation}
and
\begin{equation}
\label{deltaexp}
\delta=\det F
=\det(I+G)
=1+\tr G +\det G.
\end{equation}
In the case of incompressible motion, we have $\delta =1$, so from \eqref{deltaexp}, we get
\[
\tr G+\det G=0,
\]
and hence from \eqref{tauexp}
\[
\tau-1= \frac12\;\tr GG^\top-\det G={\mathcal O}(|G|^2).
\]
Thus, we see that for $|G|\ll1$
\[
T_2(F)={\mathcal O}[(\tau-1)|G|]={\mathcal O}(|G|^3),
\]
which produces nonlinearities which are of cubic order or higher.
Finally,  $T_3(F)$ leads to a gradient term which can be included in the pressure.
The conclusion is that the general incompressible isotropic case differs from the Hookean
case by a nonlinear perturbation which is cubic in the displacement gradient $G$.
Such terms present no further obstacles in the proof of almost global existence in 2-D,
in particular, they have the requisite symmetry properties for energy calculations, see \cite{ST07}.

\begin{thm}\label{thm-general}
Let $(v_0,G_0)\in H^k_\Lambda$, with $k\ge5$.
Suppose that $(v_0,G_0)$ satisfy the constraints \eqref{derconstr}, \eqref{incomconstr}
and $\|(v_0,G_0)\|_{H^k_\Lambda}<\epsilon$.

Assume that the smooth strain energy function $W(F)$ is isotropic, frame indifferent, and satisfies
\eqref{stressfree}, \eqref{ellipticity}.

 There exist two
positive constants $C_0$ and $\epsilon_0$ which depend only on $k$
such that, if $\epsilon \leq \epsilon_0$, then the system of incompressible isotropic
elastodynamics
\begin{equation}\label{Elas-general}
\begin{cases}
\partial_tv + v\cdot\nabla v + \nabla p = \nabla\cdot T(F),\\[-4mm]\\
\partial_tF + v\cdot\nabla F = \nabla vF,\\[-4mm]\\
\nabla\cdot v = 0,\quad \nabla\cdot F^\top=0.
\end{cases}
\end{equation}
 with initial data $(v_0, F_0)=(v_0,I+G_0)$ has a
unique solution $(v,F)=(v, I+G)$, with $(v,G) \in  H^k_\Lambda(T)$,
 $T\ge \exp(C_0/\epsilon)$, and $E_k(t)\le 2\epsilon^2$, for $0\le t<T$.

\end{thm}



\begin{thebibliography}{00}

\bibitem{Agemi00} Agemi, R. \textit{Global existence of nonlinear elastic waves}.
Invent. Math. 142 (2000), no.\ 2, 225--250.

\bibitem{Alinhac00} Alinhac, S. \textit{The null condition for quasilinear wave equations in two space
dimensions I}. Invent. Math. 145 (2001), no.\ 3, 597--618.

\bibitem{CZ2006}Chen, Y. and  Zhang, P.  \textit{The global existence of small solutions to the
incompressible viscoelastic fluid system in 2 and 3 space
dimensions}. Comm. Partial Differential Equations 31 (2006), no.\
10-12, 1793-1810.

\bibitem{Christodoulou86} Christodoulou, D. \textit{Global solutions of nonlinear hyperbolic
equations for small data}, Comm. Pure Appl. Math. 39 (1986),
267--282.

\bibitem{HX2010} He, L. and Xu, L. \textit{Global well-posedness for viscoelastic fluid system
in bounded domains}. SIAM J. Math. Anal. 42 (2010), no.\ 6,
2610-2625.

\bibitem{Hoshiga}
Hoshiga A.
\textit{The initial value problems for quasilinear wave equations in two space dimensions with small data}.
 Adv.\ Math.\ Sci.\ Appl. 5, (1995), 67Ð89.

\bibitem{HW2011} Hu, X. and Wang, D.  \textit{Global existence for the multi-dimensional
compressible viscoelastic flows}. J. Differential Equations 250
(2011), no.\ 2, 1200-1231.

\bibitem{John81}
John, F.
\textit{Blow-up for quasilinear wave equations in three space dimensions}.
Comm.\ Pure Appl.\ Math. 34 (1981), no.\ 1, 29Ð51.

\bibitem{John84} John F. \textit{Formation of singularities in elastic waves}, Lecture Notes
in Physics 195 Springer- Verlag, New York (1984), 194--210.

\bibitem{John88} John F. \textit{Almost global existence of elastic waves of finite amplitude
arising from small initial disturbances}, Comm. Pure Appl. Math.
41 (1988), 615--666.

\bibitem{JohnKlainerman84}
 John, F.\ and  Klainerman, S.
 \textit{Almost global existence to nonlinear wave equations in three space dimensions}.
 Comm.\ Pure Appl.\ Math. 37 (1984), no.\ 4, 443--455.

 \bibitem{Katayama}
  Katayama, S.
  \textit{ Global existence for systems of nonlinear wave equations in two space dimensions, II}.
  Publ.\ Res.\ Inst.\ Math.\ Sci. 31 (1995), no.\ 4, 645--665.

\bibitem{Kessenich}  Kessenich, P. \textit{Global Existence with Small Initial Data for
Three-Dimensional Incompressible Isotropic Viscoelastic
Materials}, available online at arXiv:0903.2824

\bibitem{Klainerman86} Klainerman, S. \textit{The null condition and global existence to nonlinear
wave equations}, Lect. in Appl. Math. 23 (1986), 293--326.

\bibitem{KS96} Klainerman, S. and Sideris, T. C. \textit{On almost global existence for nonrelativistic
wave equations in 3D}, Comm. Pure Appl. Math. 49 (1996), 307--321.

\bibitem{Lei} Lei, Z. \textit{Rotation-Strain decomposition for the incompressible viscoelasticity in two
dimensions}, available online at arXiv:1204.5763

\bibitem{Lei2008}  Lei, Z.  \textit{On 2D viscoelasticity with small strain}.
Arch. Ration. Mech. Anal. 198 (2010), no.\ 1, 13-37.

\bibitem{Lei3} Lei Z., Liu C.  and Zhou Y.  \textit{Global existence
for a 2D incompressible viscoelastic model with small strain}. Comm.
Math. Sci., \textbf{5} (2007), no. 3, 595--616.

\bibitem{LeiLZ08} Lei, Z., Liu, C., and  Zhou, Y. \textit{Global solutions for incompressible viscoelastic
fluids}. Arch. Ration. Mech. Anal. 188 (2008), no.\ 3, 371--398.

\bibitem{LZ2005}  Lei, Z. and Zhou, Y. \textit{Global existence of classical solutions for 2D Oldroyd model via the incompressible
limit}. SIAM J. Math. Anal., 37 (2005), no.\ 3, 797-814.

\bibitem{L95} Li, T. T. \textit{Global existence for systems of nonlinear
wave equations in two space dimensions}. Publ. RIMS, 231, 645--665
(1995).

\bibitem{LLZ2005}  Lin, F. H.,  Liu, C., and Zhang, P. \textit{On hydrodynamics of viscoelastic fluids}.
Comm. Pure Appl. Math. 58 (2005), no.\ 11 1437-1471.

\bibitem{LL2008} Lin, F.H.\ and Zhang, P.  \textit{On the initial-boundary value problem of the incompressible
viscoelastic fluid system}. Comm. Pure Appl. Math. 61 (2008), no.\
4, 539-558.

\bibitem{Qian2010} Qian, J. \textit{Well-posedness in critical
spaces for incompressible viscoelastic fluid system}. Nonlinear
Anal. 72 (2010), no.\ 6, 3222-3234.

\bibitem{Qian2011} Qian, J. \textit{Initial boundary value problems for the
compressible viscoelastic fluid}. J. Differential Equations 250
(2011), no.\ 2, 848-865.

\bibitem{QZ2010} Qian, J. and Zhang, Z. \textit{Global well-posedness for compressible viscoelastic
fluids near equilibrium}. Arch. Ration. Mech. Anal. 198 (2010),
no.\ 3, 835-868.

\bibitem{Sideris96} Sideris T. C. \textit{The null condition and global existence of nonlinear
elastic waves}, Invent. Math. 123 (1996), 323--342.


\bibitem{Sideris00} Sideris, T. C. \textit{Nonresonance and global existence of prestressed
nonlinear elastic waves}. Ann. of Math. (2) 151 (2000), no.\ 2,
849--874.
\bibitem{ST05} Sideris, T. C. and Thomases, B. \textit{Global existence for three-dimensional
incompressible isotropic elastodynamics via the incompressible
limit}. Comm. Pure Appl. Math. 58 (2005), no.\ 6, 750--788.

\bibitem{ST06} Sideris, T. C. and Thomases, B. \textit{Local energy decay for solutions of multi-dimensional
isotropic symmetric hyperbolic systems}.  J.\  Hyperbolic Differ.\ Equ. 3 (2006), no.\ 4, 673--690.

\bibitem{ST07} Sideris, T. C. and Thomases, B. \textit{Global existence for 3d
incompressible isotropic elastodynamcis}. Comm. Pure Appl. Math.
60 (2007), no.\ 12, 1707--1730.

\bibitem{SiderisTu01} Sideris, T. C. and Tu, S.-Y. \textit{Global existence for systems of nonlinear
wave equations in 3D with multiple speeds}. SIAM J. Math. Anal. 33
(2001), no.\ 2, 477--488

\bibitem{T98} Tahvildar-Zadeh, A. S.  \textit{Relativistic and nonrelativistic elastodynamics with small
shear strains}, Ann. Inst. H. Poincar${\rm \acute{e}}$ - Phys.
Th${\rm \acute{e}}$or. 69 (1998), 275--307.

\bibitem{Yokoyama}
Yokoyama, K.
\textit{Global existence of classical solutions to systems of wave equations with critical nonlinearity in three space dimensions}.
J.\ Math.\ Soc.\ Japan 52 (2000), no.\ 3, 609--632.

\bibitem{ZF-1} Zhang, T and  Fang, D.  \textit{Global well-posedness for the
incompressible viscoelastic fluids in the critical $L^p$ framework},
available at arXiv:1101.5864.

\bibitem{ZF-2} Zhang, T and  Fang, D. \textit{Global existence in critical
spaces for incompressible viscoelastic fluids}, available at
arXiv:1101.5862.

\end{thebibliography}
\end{document}